\documentclass[a4paper, 11pt]{article}

\usepackage{amsthm, amsmath, amssymb, bm}
\usepackage{enumerate}
\usepackage[colorlinks=true,
linkcolor=blue,citecolor=blue,
urlcolor=blue]{hyperref}

\renewcommand{\baselinestretch}{1.2}
\newtheorem{theorem}{Theorem}[section]
\newtheorem{corollary}{Corollary}[section]
\newtheorem{definition}{Definition}[section]
\newtheorem{conjecture}{Conjecture}

\numberwithin{equation}{section}
\numberwithin{theorem}{section}
\numberwithin{conjecture}{section}

\begin{document}

\title{\bf On the perfect $k$–divisibility of graphs}
\renewcommand\baselinestretch{1}\small

\author{David Scholz
\\ \\
\small Cologne, Germany \\
\small \tt david.scholz@pgrepds.dev
\\
\\
}

\date{\today}

\maketitle

\begin{abstract}
A graph $G$ is perfectly divisible if, for every induced subgraph $H$ of $G$, either $V(H)$ is a stable set or admits a partition into two sets $X_1$ and $X_2$ such that $\omega(H[X_1]) < \omega(H)$ and $H[X_2]$ is a perfect graph.
In this article, we propose the following generalisation of perfectly divisible graphs. A graph $G$ is perfectly $1$-divisible if $G$ is perfect and perfectly $k$-divisible if, for every induced subgraph $H$ of $G$, either $V(H)$ is a stable set or admits a partition into two sets $X_1$ and $X_2$ such that $\omega(H[X_1]) < \omega(H)$ and $H[X_2]$ is perfectly $(k-1)$-divisible, $k \in \mathbb{N}_{> 1}$. Our main result establishes that every perfectly $k$-divisible graph $G$ satisfies $\chi(G) \leq \binom{\omega(G)+k-1}{k}$ which generalises the known bound for perfectly divisible graphs.
\end{abstract}

\renewcommand\baselinestretch{1.2}\normalsize

\section{Introduction}\label{introduction}

All considered graphs are finite, simple, and undirected. We follow standard notation and terminology. Let $G$ be a graph. The \emph{complement} graph of $G$ will be denoted by $G^C$, and the \emph{complete} graph on $n$ vertices by $K_n$. The \emph{clique number} of $G$ will be denoted by $\omega(G)$, the \emph{stability number} by $\alpha(G)$, and the \emph{chromatic number} by $\chi(G)$. Let $v$ be a vertex of $G$. We write $N(v)$ for the neighbours and $M(v)$ for the nonneighbours of $v$ in $G$.

A graph $H$ is an \emph{induced subgraph} of $G$ if $V(H) \subseteq V(G)$ and $(u,v) \in E(H)$ if and only if $(u,v) \in E(G)$ for all $u,v \in V(H)$. If a graph $H$ is isomorphic to an induced subgraph of $G$, then we say that $G$ \emph{contains} $H$. If $G$ does not contain $H$, then $G$ is called $H$-\emph{free}. A class of graphs is called \emph{hereditary} if it is closed under taking induced subgraphs.

A graph $G$ is called \emph{perfect} if $\chi(H)=\omega(H)$ holds for every induced subgraph $H$ of $G$. Gy\'arf\'as \cite{GYARFAS1980} generalised the notion of perfect graphs as follows. A hereditary class $\mathcal{G}$ of graphs is $\chi$-\textit{bounded} if there is a ($\chi$-\emph{binding}) function $f: \mathbb{N} \to \mathbb{N}$ such that $\chi(G) \leq f(\omega(G)), \forall G \in \mathcal{G}$.

A graph $G$ is \textit{perfectly divisible} if, for every induced subgraph $H$ of $G$, either $V(H)$ is a stable set or admits a partition into two sets $X_1$ and $X_2$ such that $\omega(H[X_1]) < \omega(H)$ and $H[X_2]$ is perfect. The class of perfectly divisible graphs is $\chi$-bounded by $\binom{\omega + 1}{2}$ \cite{CHUDNOVSKYDIV2018}. The notion of perfectly divisible graphs was introduced by Ho\`{a}ng \cite{HOANGBANNERFREE2018}.

\section{Perfectly $k$-divisible graphs}\label{perfectKDivisibility}

We propose the following generalisation of perfectly divisible graphs.

\begin{definition}\label{def:perfectKDiv}
A graph $G$ is
    \begin{enumerate}[(i)]
        \item perfectly $1$-divisible if $G$ is perfect,
        \item perfectly $k$-divisible if, for every induced subgraph $H$ of $G$, either $V(H)$ is a stable set or admits a partition into two sets $X_1$ and $X_2$ such that $\omega(H[X_1]) < \omega(H)$ and $H[X_2]$ is perfectly $(k-1)$-divisible, $k \in \mathbb{N}_{>1}$.
    \end{enumerate}
\end{definition}

Definition \ref{def:perfectKDiv} naturally extends to the cases $k=1$ (where $G$ is perfect) and $k=2$ (where $G$ is perfectly divisible). Moreover, it is obvious that a perfectly $k$-divisible graph is also perfectly $(k+1)$-divisible. In the next theorem, we prove that a perfectly $k$-divisible graph is $\chi$-bounded.

\begin{theorem}\label{thm:generalChiBound}
Let $G$ be a perfectly $k$-divisible graph with clique number $\omega$. Then,
    \begin{equation}
        \chi(G) \leq \binom{\omega + k - 1}{k}. \label{chiBoundFormula}
    \end{equation}
\end{theorem}
\begin{proof}
We prove the claim by induction on the pair $(k, \omega) \in \mathbb{N} \times \mathbb{N}$ in lexicographic order. The base cases are $(1, \omega)$ and $(k, 1)$. If $k=1$, then $G$ is perfect and therefore $\omega$-colourable. Hence, $G$ satisfies equation \ref{chiBoundFormula}. If $\omega=1$, then $V(G)$ is a stable set. Hence, $G$ can be coloured with a single colour and again satisfies equation \ref{chiBoundFormula}.
This establishes the base cases. Now, suppose that $G$ is a perfectly $k$-divisible graph with clique number $\omega \geq 2$ and $k \geq 2$. We may assume that every perfectly $k'$-divisible graph $G'$ satisfies equation \ref{chiBoundFormula} for $k' < k$ or $k'=k$ and $\omega(G') < \omega$. Since $G$ is perfectly $k$-divisible, $V(G)$ admits a partition into two sets $X_1$ and $X_2$ such that $\omega(G[X_1]) < \omega$ and $G[X_2]$ is perfectly $(k-1)$-divisible. The pairs $(k, \omega(G[X_1])$ and $(k-1, \omega(G[X_2])$ are strictly smaller than $(k,\omega)$ in the lexicographic order. Hence, the induced subgraphs $G[X_1]$ and $G[X_2]$ satisfy equation \ref{chiBoundFormula} by the induction hypothesis. Consequently, it follows that
    \begin{align}
        \chi(G) 
                &\leq \chi(G[X_1]) + \chi(G[X_2]) \\        
                &\leq \binom{\omega + k - 2}{k} + \binom{\omega + k - 2}{k - 1} \\
                &= \binom{\omega+k-1}{k}
    \end{align}
This completes the proof.
\end{proof}

The $\chi$-binding function known for perfectly divisible graphs is included in theorem \ref{thm:generalChiBound} for the case $k=2$. In the next theorem, we prove a structural property of perfectly $k$-divisible graphs.

\begin{theorem}\label{thm:hcupk1}
If every $H$-free graph is perfectly $(k-1)$-divisible, then every $(H \cup K_1)$-free graph is perfectly $k$-divisible, $k \in \mathbb{N}_{>1}$.
\end{theorem}

\begin{proof}
Let $G$ and $H$ be graphs such that $G$ is $(H \cup K_1)$-free and let $k \in \mathbb{N}_{>1}$. We assume that every $H$-free graph is perfectly $(k-1)$-divisible and claim that $G$ is perfectly $k$-divisible. We prove this claim by induction on $\omega=\omega(G)$. The base case is trivial. So we may assume that $\omega \geq 2$ and that every $(H \cup K_1)$-free graph $G'$ with $\omega(G') < \omega$ is perfectly $k$-divisible. Let $v \in V(G)$ and $X_1=N(v)$ and $X_2=M(v)$. Since $G$ is $(H \cup K_1)$-free, $G[X_2]$ is $H$-free. Therefore, $G[X_2]$ is perfectly $(k-1)$-divisible. Since $v$ is adjacent to every vertex in $X_1$, $\omega(G[X_1]) < \omega$ holds. Thus, $G[X_1]$ is perfectly $k$-divisible by the induction hypothesis. Therefore, $X_1$ and $X_2 \cup \{v\}$ is the desired partition.
\end{proof}

Ho\`{a}ng and McDiarmid \cite{HOANGDIVISIBILITY2002} proved that a graph $G$ with stability number $\alpha$ is $\chi$-bounded by $\binom{\alpha + \omega - 1}{\alpha}$. In the next two corollaries, we establish the same $\chi$-bound.

\begin{corollary}\label{cor:stabilityNumber}
A graph with stability number $\alpha$ is perfectly $\alpha$-divisible.
\end{corollary}

\begin{proof}
Let $G$ be a graph with stability number $\alpha$. In other words, the graph $G$ is $K_{\alpha+1}^C$-free. We prove the claim by induction on $\alpha$. The base case is trivial. So we may assume that $\alpha \geq 2$ and that every $K_{\alpha}^C$-free graph is perfectly $(\alpha - 1)$-divisible by the induction hypothesis. Since $G$ is $K_{\alpha+1}^C$-free, $G$ is $(K_{\alpha}^C \cup K_1)$-free. Hence, $G$ is perfectly $\alpha$-divisible by theorem \ref{thm:hcupk1}.
\end{proof}

\begin{corollary}
Let $G$ be a graph with clique number $\omega$ and stability number $\alpha$. Then, \begin{equation}
\chi(G) \leq \binom{\omega + \alpha - 1}{\alpha}.
\end{equation}
\end{corollary}

\begin{proof}
Follows from corollary \ref{cor:stabilityNumber} and theorem \ref{thm:generalChiBound}.
\end{proof}

\section{Remarks}\label{remarks}

The graph \textit{union} $G \cup H$ of two graphs $G$ and $H$ has vertex set $V(G) \cup V(H)$ and edge set $E(G) \cup E(H)$. A $2K_2$-free graph is a graph that does not contain a $K_2 \cup K_2$ as an induced subgraph. Wagon \cite{WAGON} proved that the class of the $2K_2$-free graphs is $\chi$-bounded by $\binom{\omega + 1}{2}$. The question whether the $2K_2$-free graphs are perfectly divisible is still open.

A superclass of the $2K_2$-free graphs is the class of the $(P_n \cup P_2)$-free graphs, $n\in \{3,4\}$. Bharathi and Choudum \cite{BHARATHI2018} proved that this class is $\chi$-bounded by $\binom{\omega + 2}{3}$. Consequently, the following conjecture arises naturally.

\begin{conjecture}\label{conj1}
Every $(P_n \cup P_2)$-free graph is perfectly $3$-divisible, $n \in \{3,4\}$.
\end{conjecture}

Both the $2K_2$-free graphs and the $(P_n \cup P_2)$-free graphs share a key property that is instrumental in the proofs establishing their $\chi$-binding functions. For any graph $G$ in either class, if $X \subset V(G)$ is a subset for which there exists an edge $e$ (with $e \cap X = \emptyset$) such that no vertex in $X$ is adjacent to any vertex of $e$, then the subgraph induced by $X$ is perfect. Motivated by this result and the structural property established in theorem \ref{thm:hcupk1}, the following conjecture arises naturally.

\begin{conjecture}\label{conj2}
If every $H$-free graph is perfectly $(k-2)$-divisible, then every $(H \cup K_2)$-free graph is perfectly $k$-divisible, $k \in \mathbb{N}_{>2}$.
\end{conjecture}

Wagon \cite{WAGON} generalised the $\chi$-binding function for the $2K_2$-free graphs to an $O(\omega^{2p-2})$ $\chi$-binding function for the general class of $pK_2$-free graphs, $p \in \mathbb{N}_{>1}$. Assuming that every $2K_2$-free graph is perfectly divisible and that conjecture \ref{conj2} holds, one can easily prove the following theorem.

\begin{theorem}\label{thm:pk2}
If conjecture \ref{conj2} holds and if every $2K_2$-free graph is perfectly divisible, then every $pK_2$-free graph is perfectly $(2p-2)$-divisible, $p \in \mathbb{N}_{>1}$.
\end{theorem}

\begin{proof}
Let $G$ be a $pK_2$-free graph, $p \in \mathbb{N}_{>1}$. We prove the claim by induction on $p$. The base case follows from the assumption. So we may assume that $p > 2$ and that every $p'K_2$-free graph with $p' < p$ is perfectly $(2p'-2)$-divisible. Therefore, a $(p-1)K_2$-free graph is perfectly $(2p-4)$-divisible by the induction hypothesis. Since $G$ is $(p-1)K_2 \cup K_2$-free, it follows from conjecture \ref{conj2} that $G$ is perfectly $(2p-2)$-divisible.
\end{proof}

\subsection*{Acknowledgement}

The author would like to express his thanks to the anonymous referee for valuable suggestions.

\end{document}